\documentclass[11pt]{amsart}
\usepackage[usenames]{color}

\usepackage{latexsym}
\usepackage{amsfonts}
\usepackage{amsthm}
\usepackage{graphicx}
\usepackage{amsmath,hyperref}
\usepackage{amssymb}
\usepackage{enumitem}
\usepackage{amsmath, amsfonts, amsthm, float, amscd, latexsym, amssymb, xcolor, tikz, dsfont, cite, inputenc, faktor, mathrsfs, mathtools, hyperref, mathbbol, caption}
\usepackage{bbm}
\usepackage[all,cmtip]{xy}
\input xypic 
\usepackage{tikz}
\usepackage{todonotes}

\newcommand{\Z}{{\mathbb Z}}
\def\xa{(\underline{X},\underline{A})}

 \newtheorem{thm}{Theorem}[section]
 \newtheorem{defn}[thm]{Definition}
 \newtheorem{cor}[thm]{Corollary}
 
 \newtheorem{prop}[thm]{Proposition}
 \theoremstyle{definition}
 \theoremstyle{remark}
 \newtheorem{remark}[thm]{Remark}
 \newtheorem{ex}{Example}
 \numberwithin{equation}{section}
 
 \numberwithin{equation}{section}

\begin{document}

\title[On Polyhedral Products Spaces over Polyhedral Joins]
{On Polyhedral Products Spaces over Polyhedral Joins}

\author[E.Vidaurre]{Elizabeth Vidaurre}
\address{University of Rochester}%
\email{elizabeth.vidaurre@rochester.edu}%
\date{\today}

\begin{abstract}
The construction of a simplicial complex given by polyhedral joins (introduced by Anton Ayzenberg in \cite{AA}),  generalizes Bahri, Bendersky, Cohen and Gitler's $J$-construction and simplicial wedge construction \cite{MR3426378}. This article gives a cohomological decomposition of a polyhedral product over a polyhedral join for certain families of pairs of simplicial complexes. A formula for the Hilbert-Poincar\'{e} series is given, which generalizes Ayzenberg's formula for the moment-angle complex.
 \end{abstract}

\maketitle  

\section{Introduction}
From a simplicial complex and pairs of topological spaces, polyhedral product spaces give a family of spaces. The spaces that have the form of a polyhedral product include moment-angle complexes, complements of subspace arrangements, and intersections of quadrics among others. In certain cases, polyhedral products provide geometric realizations of right-angled Artin groups and the Stanley-Reisner ring.

One can define the \emph{polyhedral join}, $Z^*_K(\underline{X},\underline{A})$, similarly to the polyhedral product by replacing the cartesian product with the topological or simplicial join  \cite{AA}. When the pairs $(\underline{X},\underline{A})$ are simplicial complexes $(\underline{L},\underline{K})$, the associated polyhedral join is a simplicial complex. A natural question that arises concerning a polyhedral product over such a simplicial complex is what is the cohomology of the polyhedral product space over a 
polyhedral join, $Z_{Z^*_K(\underline{L},\underline{K})}(\underline{X},\underline{A})$, in terms of the simplicial complexes $K$, $K_i$, $L_i$ and the cohomology of the spaces $(\underline{X},\underline{A})$? 

Polyhedral joins behave well with polyhedral products. For example, there is the following homeomorphism between real moment-angle complexes and moment-angle complexes \cite{MR3426378}:
$$Z_{Z^*_K(\Delta^1,\partial \Delta^1)}(D^1,S^0)=Z_K(D^2,S^1)$$
This means that any moment-angle complex can be expressed as a real moment-angle complex.  This equivalence has since been used by many  \cite{Ustinovskii2011, CP, MR3084441, Lim15, MR3073929}.

Using polyhedral joins $Z^*_K(\underline{L},\underline{K})$, this can be generalized to arbitrary polyhedral products. It allows for a give and take in dimension. A polyhedral product can be expressed as another polyhedral product of either lower dimensional spaces $(\underline{X},\underline{A})$ over a higher dimensional simplicial complex, or of higher dimensional spaces $(\underline{X},\underline{A})$ over a lower dimensional simplicial complex.

We use a spectral sequence constructed by Bahri, Bendersky, Cohen and Gitler that converges to the cohomology of $Z_K(\underline{X},\underline{A})$ in terms of the long exact sequence for the pairs $(X_i,A_i)$. It gives a Kunneth-like formula for the general polyhedral product. 
In Section \ref{general} we show how the spectral sequence reduces the problem to analyzing the map $H^*(\Sigma lk_{L_i}(\sigma)) \rightarrow H^*(\Sigma lk_{K_i}(\sigma))$
induced by the inclusion of links $lk_{K_i}(\sigma) \hookrightarrow lk_{L_i}(\sigma)$,
where $\sigma$ is a simplex in $K_i$. Since there is no convenient way to describe the induced map in full generality, we consider two cases of pairs of simplicial complexes. 

In Section \ref{csc}, we restrict to the pairs $(L_i,K_i)=(\Delta^{l_i},K_i)$ where $\Delta^{l_i}$ is the $l_i$-simplex. With this pair, the associated polyhedral join is called the composed simplicial complex $K(K_1,\ldots, K_m)$. This pair covers the case of the homeomorphism between moment-angle complexes and real moment-angle complexes. As a consequence to Theorem \ref{PPoverCSC}, we give the corresponding Hilbert-Poincar\'{e} series. For the polyhedral product $Z_{K(K_1,\ldots,K_m)}(\underline{CA},\underline{A})$, we define a multigraded structure and give the multigraded Betti numbers. Moreover,  the ring structure of $H^*(Z_{Z_K(\underline{L},\underline{K})}(\underline{X},\underline{A}))$ is analyzed. 

Finally, in Section \ref{pair}, we study the case given by the pair $(L_i,K_i)=(L_i, \emptyset)$ as the map induced by the inclusion of links can be fully described. The cohomology groups will be given in Theorem \ref{empty}. 

\section{Polyhedral Product Spaces}\label{pps}
Let $[m]=\{1,2,\ldots, m\}$ denote the set of integers from $1$ to $m$. An \emph{abstract simplicial complex}, $K$, on $[m]$ is a subset of the power set of $[m]$, such that:
\begin{enumerate}
\item $\emptyset \in K$.
\item If $ \sigma \in K$ with $\tau \subset \sigma$, then $\tau \in K$.
\end{enumerate}
 A $n$-simplex is the full power set of $[n+1]$ and is denoted $\Delta^{n}$. Associated to an abstract simplicial complex is its \emph{geometric realization}, denoted $\mathcal{K}$ or $|K|$ (also called a geometric simplicial complex). A (geometric) $n$-simplex, $\Delta^n$, is the convex hull of $n+1$ points. 

We do not assume $m$ is minimal, i.e. there may exist $[n] \subsetneq [m]$ such that $K$ is contained in the power set of $[n]$. In particular, we allow \emph{ghost vertices} $\{i\} \subset [m]$ such that $\{i\} \notin K$.

Let $K$ and $L$ be simplicial complexes on sets $[m]$ and $[n]$, respectively. A map from $[m]$ to $[n]$ induces a \emph{simplicial map} from  $K$ to $L$ if it satisfies the property that simplices map to simplices.
\begin{defn}\label{full}
Let $I$ be a subset of $[m]$. The \emph{full subcomplex of $K$ in  $I$} is denoted $K_I$. It is a simplicial complex on the set $I$ and defined $$K_I := \{\sigma \in K| \sigma \subset I\}$$
\end{defn}
It will sometimes be helpful to use the notation $K|_I$ and it is often called the restriction of $K$ to $I$ in the literature.

 Given an abstract simplicial complex $K$, let $\mathcal{S}_K$ be the category with simplices of $K$ as the objects and inclusions as the morphisms. In particular, for $\sigma, \tau \in ob(\mathcal{S}_K)$, there is a morphism $\sigma \rightarrow \tau$ whenever $\sigma \subset \tau$. Define $\mathcal{CW}$ to be the usual category of CW-complexes. Define $(\underline{X},\underline{A})$ to be a collection of pairs of CW-complexes $\{(X_i,A_i)\}_{i=1}^m$, where $A_i$ is a subspace of $X_i$ for all $i$.

  \
  
{\begin{defn} \label{func}
Given an abstract simplicial complex $K$ on $[m]$, simplices $\sigma, \tau$ of $K$  and a collection of pairs of CW-complexes  $(\underline{X},\underline{A})$, define a diagram $D:\mathcal{S}_K \rightarrow \mathcal{CW}$, given by 
 $$D(\sigma)= \prod \limits_{i\in [m]}Y_i \qquad \text{  where } \displaystyle Y_i=\begin{cases} X_i & i \in \sigma \\ A_i & i \in [m]\backslash \sigma \end{cases}$$
For a morphism $f:\sigma \rightarrow \tau$, the functor $D$ maps $f$ to $i:D(\sigma)\rightarrow D(\tau)$ where $i$ is the canonical injection.

The \emph{polyhedral product space}, $ \displaystyle Z_K(\underline{X},\underline{A})\subset \prod_{i \in [m]} X_i$ is $$Z_K(\underline{X},\underline{A}):=\underset{\sigma \in K}{colim } D(\sigma)=\bigcup_{\sigma \in K} D(\sigma) $$\end{defn}}
  Notice that it suffices to take the colimit over the maximal simplices of $K$. In fact, simplicial complexes can be defined by their maximal simplices and this description will be used throughout. In the case where $(X_i,A_i)=(X,A)$ for all $i$, we write $Z_K(X,A)$. 
  
  Some examples of polyhedral products are moment-angle complexes $Z_K(D^2,S^1)$, which have the homotopy type of the complement of a subspace arrangement, and Davis-Januszkiewicz spaces $Z_K(CP^\infty, pt)$, which have a Stanley-Reisner ring as cohomology ring. For a simple example, consider the following example.
  \begin{ex} \label{helpful}
  Let $K$ be the boundary of a $2$-simplex. 
$$\begin{array}{ccl}Z_K(D^1,S^0)& = &D^1\times D^1 \times S^0 \cup D^1 \times S^0 \times D^1 \cup S^0 \times D^1 \times D^1 \\
 &=& \partial (D^1 \times D^1 \times D^1) \\
 & \cong & S^2 \end{array}$$ 
 
 In general, $Z_{\partial \Delta^{m}}(D^1,S^0) \cong S^m. $ \end{ex}
 Next we will define the polyhedral smash product, a space analogous to the polyhedral smash product with the smash product operation in place of the cartesian product.
 \begin{defn}
Let the CW-pairs $(\underline{X}, \underline{A})$ be pointed. Likewise, define a functor  $\widehat{D}(\sigma):\mathcal{S}_K \rightarrow \mathcal{CW}_*$ by:
$$\widehat{D}(\sigma)=\wedge Y_i \qquad \text{ where }\displaystyle Y_i=\begin{cases} X_i & i \in \sigma \\ A_i & i \notin \sigma \end{cases}$$ Then the \emph{polyhedral smash product} is  $$\widehat{Z}_K(\underline{X},\underline{A})=\bigcup \widehat{D}(\sigma)$$
\end{defn} 
The following theorem of Bahri, Bendersky, Cohen and Gitler (BBCG) gives a decomposition of a suspension of a polyhedral product space.
\begin{thm}[Splitting Theorem, BBCG 2008]\label{splitting}
Let $(\underline{X_I},\underline{A_I}) =\{(X_i,A_i)\}_{i\in I}$. 
Then  \begin{displaymath} \Sigma Z_K(\underline{X},\underline{A}) \simeq \Sigma \bigvee_{I\subset [m]} \widehat{Z} _{K_I}(\underline{X_I},\underline{A_I})) \end{displaymath}
\end{thm}
where  $\Sigma$ denotes the reduced suspension.
  
 Given any simplicial complex, the following procedure allows for the construction of an infinite family of associated simplicial complexes. Let $\mathcal{SC}$ be the category with simplicial complexes as the objects. 
\begin{defn}[Ayzenberg \cite{AA}]
 Let $K$ be a simplicial complex on $m$ vertices and $\sigma$ a  simplex of $K$. Let $(\underline{L},\underline{K})=\{L_i,K_i\}_{i\in [m]}$ be $m$ pairs of simplicial complexes, where $K_i$ is a subsimplicial complex of $L_i$ and both are defined on the index set $[l_i]$.
Consider a functor $D^*:\mathcal{S}_K \rightarrow \mathcal{SC}$ defined in the following way
$$D^*(\sigma)= \underset{i\in [m]}{\ast} Y_i \text{ , where } \displaystyle Y_i=\begin{cases} L_i & i \in \sigma \\ K_i & i \notin \sigma \end{cases}$$

The associated \emph{polyhedral join} is the colimit of the diagram $$Z^*_K(\underline{L},\underline{K}):=\underset{\sigma \in K}{colim D^*(\sigma)}=\bigcup_{\sigma \in K} D(\sigma)$$ 
\end{defn}

 Note that $Z^*_K(\underline{L},\underline{K})$ is a subsimplicial complex of $\underset{i\in [m]}{\ast} L_i$, which is a simplicial complex on the set  $[\sum\limits_{i\in [m]} {l_i}]$. In particular,  $D(\sigma)$ is the join of simplicial complexes $L_i$ for $i \in \sigma$ and simplicial complexes $K_j$ for $j\in [m]\backslash \sigma$. 

\begin{defn}\label{comp} Let $K$ be a simplicial complex on the set $[m]$ and $\{L_i\}_{i\in[m]}$ be simplicial complexes on the sets $[l_i]$. 

The \emph{composition of $K$ with $L_i$}, denoted $K(L_1,\ldots,L_m)$, is defined to be $$K(L_1,\ldots,L_m):= Z^*_K(\Delta^{l_i-1},L_i)$$
\end{defn}
The composition $K(L_1, \ldots, L_m)$ may also be defined by the following condition: for subsets $\sigma_i \subset [l_i]$, the set $\sigma=\sigma_1 \sqcup \ldots \sqcup \sigma_m$ is a simplex of $K(L_1,\ldots, L_m)$ whenever the set $\{ i\in [m] \: | \: \sigma_i \notin L_i \}$ is a simplex of $K$.

It is noteworthy that simplicial complexes form an operad where the simplicial complex on $m$ vertices is viewed as an $m$-adic operation. See \cite{AA} for more details.

The composition is a generalization of the J-construction \cite{MR3426378} and the simplicial wedge construction \cite{Ewald}.

The \emph{link} of $\sigma \in K$, denoted $lk_K(\sigma)$, is a simplicial complex on the set $[m] \backslash \sigma$ defined by  $\tau \in lk_K(\sigma)$ if and only if $ \sigma \cup \tau \in K$. This indexing set is used to be consistent with the definition in \cite{AA}. Given the indexing set of this complex, the link may have ghost vertices that are not ghost vertices of $K$.

\begin{ex}\label{bdayhat}
This is an example of a composition of simplicial complexes $K(L_1,\ldots, L_m)$. Let $m=3$ and $K=\{\{1\},\{2,3\}\}$.

$l_1=1$ and $L_1=\{\emptyset\}$

$l_2=1$ and $L_2=\{\{21\}\}$

$l_3=2$ and $L_3=\{\{31\},\{32\}\}$

Then 

$\begin{array}{ccccc}K(L_1,L_2,L_3) &=& D^*(\{2,3\}) & \cup  &D^*(\{1\}) \\
&=&L_1 \ast \Delta^0 \ast \Delta^1 & \cup & \Delta^0 \ast L_2 \ast L_3 \\
&=& \{\emptyset \} \ast \{21\} \ast \{31,32\}  &\cup & \{11\} \ast \{21\} \ast \{31\},\{32\} \\
&=& \{21,31,32\}, &  & \{11,21,31\}, \{11,21,32\}\end{array}$ 

\begin{figure}[H]
\centering
\includegraphics[width=3cm,height=3cm]{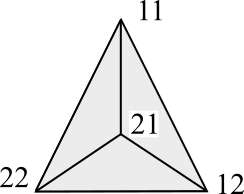}
\caption{Example of a composed simplicial complex}
\end{figure}
\end{ex}

Composed simplicial complexes have a nice relationship with polyhedral products. Note that in the following proposition the indexing set of $(\underline{X},\underline{A})$ is different than it has been up to this point, and so it will be explicitly labeled. The notation used will be $Z_K(X_i,A_i)_{i\in I}$ for some indexing set $I$.

\begin{prop}[Ayzenberg {\cite[Proposition~5.1]{AA}}]\label{aa} Let $K$ be a simplicial complex on $m$ vertices and $\{L_i\}_{i\in[m]}$ be simplicial complexes with $l_i$ vertices. We have $\displaystyle \sum_{i\in [m]}{l_i}$ pairs $(X_{ij},A_{ij})$ with $i \in [m]$ and $j \in l_i$. Then
$$Z_{K(L_1,\ldots,L_m)}(X_{ij},A_{ij})_{i\in [m],j\in[l_i]}=Z_K(\prod\limits_{j\in [l_i]} {X_{ij}},Z_{L_i}(X_{ij},A_{ij})_{j\in[l_i]})_{i\in[m]}$$
\end{prop}
These spaces are equal, not just homeomorphic. The proof inolves a shuffling of the spaces $X_i$ and $A_i$ as in the proof of proposition \ref{smash}. 

 Using methods of \cite{AA} we prove the analogous result for the polyhedral smash product.
\begin{prop}\label{smash}Let $K$ be a simplicial complex on $m$ vertices and $\{L_i\}_{i\in[m]}$ be simplicial complexes with $l_i$ vertices. Then $$\widehat{Z}_{K(L_1,\ldots,L_m)}(X_{ij},A_{ij})_{i\in [m],j\in[l_i]}=\widehat{Z}_K(\underset{j\in [l_i]}{\bigwedge} {X_{ij}},\widehat{Z}_{L_i}(X_{ij},A_{ij})_{j\in[l_i]})_{i\in[m]}$$
\end{prop}
\begin{proof}

Suppose we have a simplices $\sigma \in K$ and $\tau_i \in L_i$. Then 
$$\begin{array}{ll} \multicolumn{2}{l}{\widehat{Z}_K(\underset{j\in [l_i]}{\wedge} {X_{ij}},\widehat{Z}_{L_i}(X_{ij},A_{ij})_{j\in[l_i]})_{i\in[m]}} \\
\qquad =& \bigcup \limits_{\sigma} \big ( \underset{i\in \sigma}{\bigwedge} (\underset{j \in [l_i]}{\bigwedge} X_{ij}) \wedge \bigwedge \limits_{i\in [m]\backslash \sigma} (\bigcup \limits_{\tau_i} (\bigwedge \limits_{j \in \tau_i} X_{ij}) \wedge (\bigwedge \limits_{j\in [l_i]\backslash \tau_i} A_{ij} )) \big ) \\
\qquad =&\bigcup \limits_{\sigma, \tau_i} \big ( \underset{i\in \sigma}{\bigwedge} (\underset{j \in [l_i]}{\bigwedge} X_{ij}) \wedge \bigwedge \limits_{i\in [m]\backslash \sigma} ((\bigwedge \limits_{j \in \tau_i} X_{ij}) \wedge (\bigwedge \limits_{j\in [l_i]\backslash \tau_i} A_{ij} )) \big ) \end{array}$$

Since simplices in $\sigma' \in K(L_1,\ldots, L_m)$ are of the form $\sigma'= (\bigcup \limits_{i\in [m] \backslash \sigma} \tau_i) \cup (\bigcup \limits_{i \in \sigma} \Delta^{l_i-1})$, $$\underset{i\in \sigma}{\bigwedge} (\underset{j \in [l_i]}{\bigwedge} X_{ij}) \wedge \bigwedge \limits_{i\in [m]\backslash \sigma} (\bigwedge \limits_{j \in \tau_i} X_{ij}) = \bigwedge \limits_{ij \in \sigma'} X_{ij}$$

Note that  $\bigcup \limits_{i\in [m] \backslash \sigma} [l_i] \backslash \tau_i = [\Sigma l_i]\backslash \sigma'$, and thus $$ \bigwedge \limits_{i\in [m]\backslash \sigma} (\bigwedge \limits_{j\in [l_i]\backslash \tau_i} A_{ij} )= \bigwedge \limits _{ij\in [\Sigma l_i]\backslash \sigma'} A_{ij}$$
And finally, 
$$\bigcup \limits_{\sigma'}( \bigwedge \limits_{ij \in \sigma'} X_{ij} \wedge \bigwedge \limits _{ij\in [\Sigma l_i]\backslash \sigma'} A_{ij})
=\widehat{Z}_{K(L_1,\ldots,L_m)}(X_{ij},A_{ij})_{i\in [m],j\in[l_i]} 
$$
\end{proof}
Similarly, we can make the same type of argument for the polyhedral join in place of the composed simplicial complex for the most general form.
\begin{thm}\label{genjoin} Given $m$ pairs of simplicial complexes $(\underline{L},\underline{K})$ where $L_i$ and $K_i$ are simplicial complexes on the vertex set $[l_i]$ ($L_i$ may have ghost vertices). Taking $\displaystyle \sum_{i\in [m]} l_i$ pairs $(X_{ij},A_{ij})$, we have
$$Z_{Z^*_K(\underline{L},\underline{K})}(X_{ij},A_{ij})_{i\in [m], j \in [l_i]}=Z_K(Z_{L_i}(X_{ij},A_{ij})_{j\in[l_i]},Z_{K_i}(X_{ij},A_{ij})_{j\in[l_i]})_{i\in [m]}  $$
\end{thm}
 The complexes $K_i$ must have an indexing set of the same cardinality of the indexing set of $L_i$; otherwise, the statement is not true. In particular, $L_i$ may have ghost vertices. Keep in mind that including ghost vertices does change the polyhedral product by multiplying by $A_i$ where $i$ is a ghost vertex. 
 
  Using Example \ref{helpful}, we can see how a moment-angle complex can be expressed as a real moment-angle complex.
 \begin{cor}[Bahri, Bendersky, Cohen and Gitler \cite{MR3426378} ] Real moment-angle complexes over a composed simplicial complex are homeomorphic to moment-angle complexes:
$$Z_{K(\partial \Delta_{1},\ldots,\partial \Delta_{1})}(D^1,S^0)=Z_K(D^2,S^1)$$
\end{cor} 

\subsection{The BBCG spectral sequence}
Recall that our goal is to compute the cohomology of $Z_{Z^*_K(\underline{L},\underline{K})}(\underline{X},\underline{A})$ in terms of  $K$, $K_i$, $L_i$, $H^*(X_{ij})$ and $H^*(A_{ij})$. To do so, we will use a spectral sequence developed by BBCG \cite{BBCG4}. It gives a Kunneth-like formula for the cohomology of a polyhedral product as long as the pairs $(\underline{X},\underline{A})$ satisfy the following freeness condition.

\begin{defn} \label{free}
Given the pair $(X_i,A_i)$, the associated long exact sequence is given by
$$  \ldots \overset{\delta}{\rightarrow} \widetilde{H}^*(X_i/A_i) \overset{g}{\rightarrow} H^*(X_i) \overset{f}{\rightarrow} H^*(A_i) \overset{\delta}{\rightarrow} \widetilde{H}^{\ast+1}(X_i/A_i) \overset{g}{\rightarrow} \ldots $$
Assume that the cohomology groups of the pair have the following decomposition
$$\begin{array}{ccccc}
H^*(A_{i})&=&B_{i}\oplus E_{i}\\
H^*(X_{i})&=&B_{i}\oplus C_{i}\\
\widetilde{H}^*(X_i/A_i)&=& C_i \oplus W_i\\
 \end{array}$$
where $W_i$ is $sE_i$, the suspension of $E_i$. Additionally, assume $1\in B_i$, and for $b\in B_i, c \in C_i, e \in E_i, w \in W_i=sE_i$, we have 
\begin{center} $b \overset{f}{\mapsto} b \overset{\delta}{\mapsto} 0, \qquad c \overset{g}{\mapsto} c \overset{f}{\mapsto} 0, \qquad e \overset{\delta}{\mapsto} w \overset{g}{\mapsto} 0$ \end{center} 
\end{defn}
Before defining the spectral sequence, we will give some notation and recall the definition of a half smash product:
\begin{enumerate}
\item for $\sigma =\{i_1,\ldots, i_k\}$, define $\widehat{X}^{ \sigma} := X_{i_1} \wedge \ldots \wedge X_{i_k}$  and ${A}^{\sigma}=A_{i_1} \times \ldots \times A_{i_k}$
\item the complement of a set $\sigma \subset [m] $ is  $\sigma^c = [m] \backslash \sigma$ 
\item given a basepoint $x_0 \in X$, the right half smash product $X \rtimes Y = (X \times Y) / (x_0 \times Y)$
\item for a subset $I$ and a simplex $\sigma$ such that $\sigma \subset I$, define $$Y^{I, \sigma}:= \bigotimes \limits_{i\in \sigma} C_i \otimes \bigotimes \limits_{i \in I-\sigma} B_i$$
\end{enumerate}

Choosing a lexicographical ordering for the simplices of $K$ gives a filtration of the associated polyhedral product space and polyhedral smash product, which in turn leads to a spectral sequence converging to the reduced cohomology of $Z_K(\underline{X},\underline{A})$ and a spectral sequence converging to the reduced cohomology of $\widehat{Z}_K(\underline{X},\underline{A})$. The $E_1^{s,t}$ term for $Z_K(\underline{X},\underline{A})$ has the following description.
\begin{thm}[Bahri, Bendersky, Cohen and Gitler \cite{BBCG4}]\label{bbcgss} 
 There exist spectral sequences $$E^{s,t}_r \rightarrow H^*(Z_K(\underline{X},\underline{A}))$$ $$\widetilde{E}^{s,t}_r \rightarrow H^*(\widetilde{Z}_K({X},{A}))$$ with $ E^{s,t}_1 = \widetilde{H}^t((\widehat{X/A})^{ \sigma} \rtimes {A}^{\sigma^c})$ and $\widetilde{E}^{s,t}_1 = \widehat{Z}_K(\underline{X},\underline{A}) =  \widetilde{H}^t((\widehat{X/A})^{ \sigma} \wedge \widehat{A}^{ \sigma^c})$ where  $s$ is index of $\sigma$ in the lexicographical ordering and the differential $d_r : E_r^{s,t} \rightarrow E_r^{s+r,t+1}$ is induced by the coboundary map $\delta : E \rightarrow W=sE$. Moreover, the spectral sequence is natural for embeddings of simplicial maps with the same number of vertices and with respect to maps of pairs. The natural quotient map $$Z_K(\underline{X},\underline{A}) \rightarrow \widehat{Z}_K(\underline{X},\underline{A})$$ induces a morphism of spectral sequences and the stable splitting of Theorem \ref{splitting} induces a morphism of spectral sequences. \end{thm}

Following \cite{BBCG4}, Definition \ref{free} and the K{\"u}nneth theorem imply that the $E_1$ terms of the spectral sequence for $\widehat{Z}_K(\underline{X},\underline{A})$ decompose as a direct sum of spaces $W^N \otimes C^S \otimes B^T \otimes E^J$ such that $N \cup S =\sigma$, $T \cup J = \sigma^c$ and $N, S, J, T$ are disjoint. We have that $S$ is a simplex in $K$ as $N \cup S$ is a simplex in $K$. Since the differential is induced by the coboundary $\delta : E \rightarrow W$, consider all the possible summands $W^N \otimes C^S \otimes B^T \otimes E^J$ for $S$ and $T$ are fixed. It must be the case that $N$ is a simplex in $K$ and that $N$ is a subset of $[m] \backslash (S \cup T)$. Therefore all such $N$ correspond to simplices in the link of $S$ in $K$ restricted to the vertex set $[m] \backslash (S \cup T)$.

  \begin{thm}[Bahri, Bendersky, Cohen and Gitler \cite{BBCG4}]\label{BBCGSS}
Let $(\underline{X},\underline{A})$ satisfy the decomposition described in Definition \ref{free} $$\begin{array}{ccccc}
H^*(A_{i})&=&B_{i}\oplus E_{i}\\
H^*(X_{i})&=&B_{i}\oplus C_{i}\\
 \end{array}$$
Then $$H^*(Z_K(\underline{X},\underline{A})) = \bigoplus_{I\subset [m], \sigma\subset I} E^{I^c}\otimes Y^{I, \sigma}\otimes \widetilde{H}^*(\Sigma| lk(\sigma)_{I^c}|) $$

where:

\begin{enumerate}
\item $\sigma$ is a simplex in $K$,
  \item  $lk(\sigma)_{I^c}=\{\tau \subset [m]\backslash I \:|\: \tau \cup \sigma \in K\}$ is the link of $\sigma$ in $K$ restricted to the set $[m]\backslash I$,
 \item $Y^{I, \sigma}= \bigotimes \limits_{i\in \sigma} C_i \otimes \bigotimes \limits_{i \in I-\sigma} B_i$, and
 \item $\widetilde{H}^*(\Sigma \emptyset)=1$.
\end{enumerate}
\end{thm}
\begin{thm}[Bahri, Bendersky, Cohen and Gitler \cite{BBCG4}]
Let $$\begin{array}{ccccc}
\widetilde{H}^*(A_{i})&=&\widetilde{B}_{i}\oplus E_{i}\\
\widetilde{H}^*(X_{i})&=&\widetilde{B}_{i}\oplus C_{i}\\
 \end{array}$$
Then $$H^*(\widehat{Z}_K(\underline{X},\underline{A})) = \bigoplus_{I\subset [m], \sigma\subset I} E^{I^c}\otimes Y^{I, \sigma}\otimes \widetilde{H}^*(\Sigma| lk(\sigma)_{I^c}|) $$

where:
\begin{enumerate}
\item $\sigma$ is a simplex in $K$,
 \item $lk(\sigma)_{I^c}=\{\tau \subset [m]\backslash I \:|\: \tau \cup \sigma \in K\}$ is the link of $\sigma$ in $K$ restricted to the set $[m]\backslash I$, 
 \item $Y^{I, \sigma}= \bigotimes \limits_{i\in \sigma} C_i \otimes \bigotimes \limits_{i \in I-\sigma} \widetilde{B}_i$ where $\widetilde{B}_i=B_i \backslash \{1\}$, and
 \item $\widetilde{H}^*(\Sigma \emptyset)=1$.
\end{enumerate}
\end{thm}
Consequently, assuming that the cohomology of the pairs $(X_{ij},A_{ij})$ satisfy the freeness condition, then once we the kernel, image and cokernel of the pair $(Z_{L_i}(\underline{X},\underline{A}),Z_{K_i}(\underline{X},\underline{A}))$ is computed, this theorem can be applied.

\section{The General Polyhedral Join}\label{general}
 In this section we aim to understand the cohomology of the space $Z_{Z^*_K(\underline{L},\underline{K})}(\underline{X},\underline{A})$. To do so we will use Theorem \ref{join}, and therefore will be applying Theorem \ref{BBCGSS} to the space $Z_K(Z_{L_i}(X_{ij},A_{ij})_{j\in[l_i]},Z_{K_i}(X_{ij},A_{ij})_{j\in[l_i]})_{i\in [m]} $. Fix an $i$.  The inclusion $K_i \hookrightarrow L_i$ induces the obvious inclusion $Z_{K_i}(\underline{X},\underline{A}) \hookrightarrow Z_{L_i}(\underline{X},\underline{A})$, which in turn induces a map in cohomology $$H^*(Z_{L_i}(\underline{X},\underline{A})) \overset{\phi}{\rightarrow} H^*(Z_{K_i}(\underline{X},\underline{A}))$$
By theorem \ref{BBCGSS}, we have that
$$\begin{array}{ccl}  H^*(Z_{L_i}(\underline{X},\underline{A})) &=& \bigoplus \limits_{\sigma \in L_i, \sigma\subset I\subset [l_i]} E^{I^c}\otimes Y^{I, \sigma}\otimes \widetilde{H}^*(\Sigma| lk_{L_i}(\sigma)_{I^c}|)  \\
& =& \big (\bigoplus \limits_{\tau \in K_i, \tau \subset I \subset [l_i]} E^{I^c}\otimes Y^{I, \tau}\otimes \widetilde{H}^*(\Sigma| lk_{L_i}(\tau)_{I^c}|)\big)    \\
& &\multicolumn{1}{c}{\oplus} \\
& & \big(\bigoplus \limits_{\sigma \notin K_i, \sigma \subset I\subset [l_i]} E^{I^c}\otimes Y^{I, \sigma}\otimes \widetilde{H}^*(\Sigma| lk_{L_i}(\sigma)_{I^c}|) \big) \end{array}$$
\noindent and
$$ H^*(Z_{K_i}(\underline{X},\underline{A})) = \bigoplus_{\tau \in K_i, \tau \subset I\subset [l_i]} E^{I^c}\otimes Y^{I, \tau}\otimes \widetilde{H}^*(\Sigma| lk_{K_i}(\tau)_{I^c}|) $$
\noindent where  $lk_{L_i}(\sigma)_{I^c}$ is the link of $\sigma$ in $L_i$ restricted to the vertex set $[l_i]\backslash I$. Therefore, understanding the map $\phi$ is reduced to finding the image of the factor $\alpha \in  \widetilde{H}^*(\Sigma| lk_{L_i}(\tau)_{I^c}|)$ for $\tau \in K_i$. 
 Recall (from the discussion before theorem \ref{BBCGSS}) that $\alpha$ corresponds to the exponent of the $W$s in a summand $E^J \otimes W^N \otimes C^\tau \otimes B^{I\backslash \tau}$ of
 $$\widetilde{H}^*(\widehat{X/A}^{ (N\cup \tau)}) \otimes \widetilde{H}^*(\widehat{A}^{ (N\cup \tau)^c})$$ 
 in the $E_1$ page of the BBCG spectral sequence. 
If $N\cup \tau$ is a simplex in $K_i$, then \newline
$\widetilde{H}^*(\widehat{X/A}^{ (N\cup \tau)}) \otimes \widetilde{H}^*(\widehat{A}^{ (N\cup \tau)^c})$ in the spectral sequence of $K_i$ maps to 
\newline $\widetilde{H}^*(\widehat{X/A}^{ (N\cup \tau)}) \otimes \widetilde{H}^*(\widehat{A}^{ (N\cup \tau)^c})$ in the spectral sequence of $Z_{L_i}(\underline{X},\underline{A})$ by the naturality of the spectral sequence for embeddings of simplicial maps. 
In particular, if $N\cup \tau \in K_i$, then $$ \phi : E^J \otimes W^N \otimes C^\tau \otimes B^{I\backslash \tau} \mapsto E^J \otimes W^N \otimes C^\tau \otimes B^{I\backslash \tau} $$ and if $N\cup \tau \notin K_i$, then $$ \phi : E^J \otimes W^N \otimes C^\tau \otimes B^{I\backslash \tau} \mapsto 0$$ This is the dual of the inclusion $$ lk_{K_i}(\sigma)_{I^c} \hookrightarrow  lk_{L_i}(\sigma)_{I^c}$$
 In other words, the map $\phi$ is induced by the inclusion $ lk_{K_i}(\sigma)_{I^c} \hookrightarrow  lk_{L_i}(\sigma)_{I^c}$

\section{Composed simplicial complexes}\label{csc}

Recall that a composed simplicial complex is the polyhedral join $Z^*_K(\Delta^{l_i-1},L_i)$, denoted $K(L_1, \ldots, L_m)$, which is of particular importance because it yields the equivalence between moment-angle complexes and some real moment-angle complexes. In this case, the link of any simplex in $\Delta^{l_i-1}$ is a simplex, and hence the geometric realization of the link is contractible. This means the cohomology of the suspension of the link is trivial, and therefore so is a summand $E^{I^c}\otimes Y^{I,\sigma} \otimes \widetilde{H}^*(\Sigma |lk(\sigma)_{I^c}|)$ unless $I=[m]$. In particular, the image, kernel and cokernel of the map \begin{equation}\label{map}H^*(\prod\limits_{j\in [l_i]}X_{ij}) \rightarrow H^*(Z_{L_i}(X_{ij},A_{ij}))\end{equation} can be computed. The following proposition gives the cohomology of the polyhedral product over a composed simplicial complex in terms of the cohomology of the pairs and the simplicial complexes (using notation described after Definition \ref{free}).

\

 \begin{thm}\label{PPoverCSC}
 Given $m$ simplicial complexes $\{L_i\}_{i\in [m]}$, where $L_i$ is a complex on the set $[l_i]$. For each $i\in [m]$, let there be a family of pairs $(X_{ij},A_{ij})_{j\in[l_i]}$ satisfying Definition \ref{free}.
$$\begin{array}{ccccc}
H^*(A_{ij})&=&B_{ij}\oplus E_{ij}\\
 H^*(X_{ij})&=&B_{ij}\oplus C_{ij}
 \end{array}$$
 Then 

\noindent
$H^*(Z_{K(L_1,\ldots,L_m)}(X,A))=$
 $$  \bigoplus_{I,\sigma, J, \tau, \rho, \rho'}   \widetilde{H}^*(\Sigma | lk(\sigma)_{I^c}|)\otimes \Big(\bigotimes\limits_{k\in [m]\backslash I}E^{J^c}\otimes \widetilde{H}^*(\Sigma |lk(\tau)_{J^c}|)\otimes Y^{J,\tau} \Big) \otimes \Big (\bigotimes\limits_{s \in \sigma}Y^{[l_s],\rho'}\otimes \bigotimes\limits_{s'\in I\backslash \sigma}Y^{[l_{s'}],\rho} \Big) $$


where: 

\begin{enumerate}
\item $I\subset [m]$ and $\sigma $ is a simplex in $ K$, with  $\sigma\subset I$,
\item $J\subsetneq [l_k]$ and   $\tau$ is a simplex in $ L_k$,  with $\tau \subset J$,
\item $\rho'$ is a nonsimplex in $L_s$  and $\rho$ is a simplex in $L_{s '}$,
\end{enumerate}
	



\end{thm}
\begin{proof}
Recall Ayzenberg's Theorem, $$Z_{K(L_1,\ldots,L_m)}(X_{ij},A_{ij})=Z_K(\prod\limits_{j\in [l_i]}X_{ij},Z_{L_i}(X_{ij},A_{ij}))$$  
and that we want to find the kernel, $C_i$, image, $B_i$, and cokernel, $E_i$, of the induced map $H^*(\prod\limits_{j\in [l_i]}X_{ij}) \rightarrow H^*(Z_{L_i}(X_{ij},A_{ij}))$.

\begin{tabular}{ll}Since $H^*(\prod \limits_{j=1}^{l_i}X_{ij})$&$=\bigotimes\limits_{j=1}^{l_i}(B_{ij}\oplus C_{ij})$\\
&$=\bigoplus\limits_{\rho \in \Delta^{l_i-1}}Y^{[l_i],\rho}$\\
&$=\bigoplus\limits_{\rho \in L_i}Y^{[l_i],\rho} \oplus \bigoplus\limits_{\rho'\notin L_i}Y^{[l_i],\rho'}$
\end{tabular}

\noindent and by theorem \ref{BBCGSS}
$$H^*(Z_{L_i}(X_{ij},A_{ij})_{j\in [l_i]})=\bigoplus \limits_{\rho\in L_i}Y^{[l_i],\rho} \oplus \bigoplus\limits_{J\subsetneq [l_i],\tau\subset J}E^{[l_i]-J}\otimes \widetilde{H}^*(\Sigma |lk(\tau)_{J^c}|)\otimes Y^{J,\tau}  $$ 
we have  
\begin{align}
 B_i&=\bigoplus \limits_{\rho\in L_i}Y^{[l_i],\rho}   \label{Bi}  \\
\label{Ci} C_i&=\bigoplus\limits_{\rho'\notin L_i}Y^{[l_i],\rho'}\\
\label{Ei}  E_i&= \bigoplus\limits_{\tau \in L_i,\tau\subset J\subsetneq [l_i]}E^{[l_i]-J}\otimes \widetilde{H}^*(\Sigma | lk(\tau)_{J^c}|)\otimes Y^{J,\tau}
 \end{align}

Notice that $C_i$ is the Stanley-Reisner ideal of $L_i$, $I(L_i)$, and $B_i$ is the Stanley-Reisner ring of $L_i$, $SR(L_i)$ (see Corollary \ref{SRversion}). Now applying Theorem \ref{BBCGSS} again, 
$$H^*(Z_K(\prod\limits_{j\in [l_i]}X_{ij},Z_{L_i}(X_{ij},A_{ij})))= \bigoplus_{I\subset [m], \sigma\subset I} E^{I^c}\otimes Y^{I, \sigma}\otimes \widetilde{H}^*(\Sigma| lk(\sigma)_{I^c}|) $$

Substituting (\ref{Ei}) and expanding, 

\begin{tabular}{cl}$E^{[m]-I}$ &$=\bigotimes \limits_{k\in {[m]-I}} E_k $\\
&$=\bigotimes \limits_{k\in {[m]-I}}\bigoplus\limits_{J, \tau} E^{[l_k]-J}\otimes \widetilde{H}^*(\Sigma | lk(\tau)_{J^c}|)\otimes Y^{J,\tau}$\\
& $=\bigoplus\limits_{J, \tau} \Big(\bigotimes \limits_{k\in {[m]-I}} E^{[l_k]-J}\otimes \widetilde{H}^*(\Sigma | lk(\tau)_{J^c}|)\otimes Y^{J,\tau}$ \Big)\\
\end{tabular}

\

Next, substituting (\ref{Bi}) and (\ref{Ci}) and expanding,

\begin{tabular}{cl}$Y^{I,\sigma}$&$=\bigotimes\limits_{s\in\sigma}C_s\otimes \bigotimes \limits_{s'\in I-\sigma}B_{s'}$\\
&$=\bigotimes\limits_{s\in\sigma}(\bigoplus\limits_{\rho' \notin L_s}Y^{[l_s],\rho'}) \otimes \bigotimes \limits_{s'\in I-\sigma}(\bigoplus\limits_{\rho\in L_{s'}}Y^{[l_{s'}],\rho})$\\
&$=\bigoplus\limits_{\rho' \notin L_s}(\bigotimes\limits_{s\in\sigma}Y^{[l_s],\rho'}) \otimes \bigoplus\limits_{\rho\in L_{s'}}(\bigotimes \limits_{s'\in I-\sigma}Y^{[l_{s'}],\rho})$ \\
&$=\bigoplus\limits_{\rho', \rho} \Big( \bigotimes\limits_{s\in \sigma} Y^{[l_s],\rho'}\otimes \bigotimes\limits_{s'\in I-\sigma}Y^{[l_{s'}],\rho} \Big)$
\end{tabular}

\

The proposition follows.


\end{proof}
Recall that the decomposition of $H^*(Z_K(\underline{X},\underline{A}))$ in Theorem \ref{bbcgss} differs from the decomposition for $H^*(\widehat{Z}_K(\underline{X},\underline{A}))$ by the presence of $1\in B_{ij}$. As a consequence, the same proof provides a decomposition for $H^*(\widehat{Z}_{K(L_1,\ldots,L_m)}(\underline{X},\underline{A}))$. Below is an example of how Proposition \ref{PPoverCSC} can be used to compute the Poincar\'{e} series for the cohomology of a polyhedral product over a composed simplicial complex.
\begin{ex}
 Suppose we have that $K$ is two ghost vertices, $L_1$ is a point and $L_2$ is two points. Namely, $K=\{ \emptyset \}$ is indexed by $[2]$, $L_1=\{\{11\}\}$, $L_2=\{\{21\},\{22\}\}$. Given a pair of spaces such that $\widetilde{H}^*(X)=\langle b_4,c_6 \rangle$ and $\widetilde{H}^*(A)=\langle e_2,b_4 \rangle$, we will compute the Poincar\'{e} series for $H^*(\widehat{Z}_{K(L_1,L_2)}(X,A))$.
\begin{enumerate}
\item For $I=\emptyset$, the only possible simplex, $\sigma$, is the empty set. Since $J_1\neq [11]$, the only choice for $J_1$ and $\tau_1$ is the empty set, then $lk(\tau_1)_{J_1^c}=L_1$ is contractible. There is no contribution to the Poincare series. This means we will only consider $I$ such that $1\in I$. 
 \item In the case $I=\{1\}$ and  $\sigma = \emptyset$, we consider subsets, $J_k$, of $L_k$ for $k\in [m]\backslash I$, and simplices $\rho'_s \in L_s$ and $\rho_{s'} \in L_{s'}$ for $s\in \sigma$ and $s' \in I \backslash \sigma$.  
	\begin{enumerate}
 	\item For $J_2=\emptyset$ and  $\tau_2 = \emptyset$,  $lk(\tau_2)_{J_2^c}=L_2$ and  $\widetilde{H}^*(\Sigma |lk(\tau_2)_{J_2^c}|)=<\iota_1>$
		\begin{enumerate}
		\item If $\rho_1=\emptyset$, then  $\overline{P}(E_{21} \otimes E_{22} \otimes \iota_1 \otimes B_{11}) =t^9$ is contributed to the Poincare series
		\item If $\rho_1=\{11\}$, then  $\overline{P}(E_{21} \otimes E_{22} \otimes \iota_1 \otimes C_{11}) =t^{11}$ is contributed
		\end{enumerate}
	\item Similarly, with $J_2=\{21\}$ and  $\tau_2= \emptyset$, $lk(\tau_2)_{J_2^c}=\{\{22\} \}$ is contractible
	\item If $J_2=\{21\}$ and  $\tau_2= \{21\}$, then  $lk(\tau_2)_{J_2^c}=\emptyset$
		\begin{enumerate}
 		\item If $\rho_1=\emptyset$, then  $\overline{P}(C_{21} \otimes E_{22} \otimes 1 \otimes B_{11}) =t^{12}$
		\item If $\rho_1=\{11\}$, then  $\overline{P}(C_{21} \otimes E_{22} \otimes 1 \otimes C_{11}) =t^{14}$
		\end{enumerate}
	\item If $J_2=\{22\}, \tau_2= \{22\}$ then  $lk(\tau_2)_{J_2^c}=\emptyset$
		\begin{enumerate}
		\item If $\rho_1=\emptyset$, then  $\overline{P}(C_{22} \otimes E_{21} \otimes 1 \otimes B_{11}) =t^{12}$
		\item If $\rho_1=\{11\}$, then  $\overline{P}(C_{22} \otimes E_{21} \otimes 1 \otimes C_{11}) =t^{14}$
		\end{enumerate}
	\end{enumerate}
 \item For $I=\{2,1\}$ and  $\sigma = \emptyset$, $lk(\sigma)_I=\emptyset$
		\begin{enumerate}
		\item For $\rho_1 = \emptyset$,
			\begin{enumerate}
			\item If $\rho_2 =\emptyset $, then $\overline{P}(B_{11} \otimes B_{21} \otimes B_{22})=t^{12}$
			\item If $\rho_2 = \{21\}$, then $\overline{P}(B_{11} \otimes C_{21} \otimes B_{22}) = t^{14}$
			\item If $\rho_2 = \{22\}$, then $\overline{P}(B_{11} \otimes B_{21} \otimes C_{22})= t^{14}$
			\end{enumerate}
		\item For $\rho_1 = \{11\}$
			\begin{enumerate}
			\item If $\rho_2 =\emptyset $, then $\overline{P}(C_{11} \otimes B_{21} \otimes B_{22})=t^{14}$
			\item If $\rho_2 = \{21\}$, then $\overline{P}(C_{11} \otimes C_{21} \otimes B_{22}) = t^{16}$
			\item If $\rho_2 = \{22\}$, then $\overline{P}(C_{11} \otimes B_{21} \otimes C_{22})= t^{16}$
			\end{enumerate}
		\end{enumerate}
\end{enumerate}
In conclusion, $$\overline{P}(H^*(\widehat{Z}_{K(L_1,L_2)}(X,A)))=t^9 +t^{11} +3t^{12} + 5t^{14} +2t^{16}$$
Since 
$$\begin{array}{cl} K(L_1,L_2) & = L_1 \ast L_2 \\
&=\{\{11,21\}, \{11,22\}\} \end{array}$$ a simplicial complex with three vertices and two edges, we can see that this is consistent with Example 5.8 in \cite{BBCG4}. Their example computes the Poincare series for the polyhedral product over a simplicial complex with three vertices and two edges, and spaces with equivalent cohomology. We obtained the same answer using a different method.
\end{ex}
 \begin{defn}\label{srdef}
Let $K$ be a simplicial complex on $m$ vertices and $k$ be a ring. Consider $k[m]:=k[v_1, \ldots, v_m]$, the ring of polynomials in $m$ indeterminants. The \emph{generalized Stanley-Reisner ideal} of $K$, $I(K)$, is generated by square-free monomials indexed by the non-simplices of $K$ $$I(K)=<v_{i_1}\ldots v_{i_n} \:|\: \{i_1,\ldots,i_n\} \notin K >$$ The Stanley-Reisner (or face) ring of a simplicial complex, $K$, is denoted $k[K]$ and is defined as $$\displaystyle k[K] = \faktor{k[m]}{I(K)}$$	  
\end{defn}

The following is another version of Theorem \ref{PPoverCSC} which highlights the role of the Stanley-Reisner ring (see Definition \ref{srdef}).
\begin{cor}\label{SRversion}
Following Definition \ref{free}, suppose we have a decomposition
$$
\begin{array}{ccccc}
H^*(A_{ij})&=&B_{ij}\oplus E_{ij}\\
H^*(X_{ij})&=&B_{ij}\oplus C_{ij}
 \end{array}$$ with $E_{ij}$ the cokernel of $H^*(X_{ij})\rightarrow H^*(A_{ij})$, $B_{ij}$ the image, and $C_{ij}$ the kernel. Then 
  
\noindent
$H^*(Z_{K(L_1,\ldots,L_m)}(X,A))=$
 $$\bigoplus\limits_{I,\sigma, J, \tau}\widetilde{H}^*(\Sigma | lk(\sigma)_{I^c}|)\otimes \Big(\bigotimes\limits_{k\in [m] \backslash I}E^{[l_k]-J}\otimes \widetilde{H}^*(\Sigma | lk(\tau)_{J^c}|)\otimes Y^{J,\tau} \Big)\otimes \Big(\bigotimes\limits_{s \in \sigma}I(L_s)\otimes \bigotimes\limits_{s' \in I \backslash \sigma}SR(L_{s'}) \Big)$$

where 
\begin{enumerate}
\item $I\subset [m]$ and $\sigma $ is a simplex in $ K$, with  $\sigma\subset I$
\item   $J\subsetneq [l_k]$.   $\tau$ is a simplex in $ L_k$,  with $\tau \subset J$
 \item  $\rho'$ is a nonsimplex in $L_s$  and $\rho$ is a simplex in $L_{s '}$..
\end{enumerate}
\end{cor}

\begin{cor}
Let $I'=I'_1 \sqcup \ldots \sqcup I'_m \subset [\Sigma{l_i}]$ and $\sigma'=\sigma'_1 \sqcup \ldots \sqcup \sigma'_m$ a simplex in $K(L_1,\ldots,L_m)$ with $\sigma'\subset [\sum{l_i}] \backslash I'$. Then 
$$\widetilde{H}^*(\Sigma |lk(\sigma')_{I'}|)=\bigotimes \limits_{k\in I}\widetilde{H}^*(\Sigma | lk(\sigma'_k)_{I'_k }|)\otimes \widetilde{H}^*(\Sigma | lk(\sigma)_{I}|)$$
where $I=\{i\in [m]|\: I'_i \neq \emptyset\}$ and $\sigma=\{i\in [m] \backslash I|\sigma'_i\notin L_i\}$.
\end{cor}
\begin{proof}
By theorem \ref{splitting}, $H^*(Z_{K(L_1,\ldots,L_m)}(X,A))=\bigoplus\limits_{\sigma',I'} E^{[\sum l_i] - I'}\otimes \widetilde{H}^*(\Sigma | lk(\sigma')_{I'^c}|)\otimes Y^{I',\sigma'}$. Moreover, by proposition \ref{PPoverCSC}, we have that $E^{[\sum{l_i}]-I'}= \bigotimes\limits_{k\in [m]-I}E^{[l_k]-J_k}$ whenever 
$I'_i=[l_i]$ for all $i\in I$ and $I'_k=J_k$ for all $k\in [m]-I$. We also have that $Y^{I',\sigma'}=\bigotimes\limits_{k\in [m]-I}Y^{J_k,\tau}\otimes (\bigotimes\limits_{s \in \sigma}Y^{[l_s],\rho'})\otimes (\bigotimes\limits_{s'\in I-\sigma} Y^{[l_{s'}],\rho}) $ whenever $\sigma'=(\cup \tau)\cup (\cup \rho') \cup (\cup \rho)$. Thus for $k \in [m]-I$, $\sigma_k'=\tau$. Similarly, for $s\in \sigma$, $\sigma_s'=\rho'\notin L_s$. In other words $s\in \sigma$ if and only if $\sigma_s' \notin L_s$. Lastly, for $s'\in I-\sigma$, $\sigma'_{s'}=\rho$. A change of notation is used so that the proposition is not stated in terms of complements of sets.
\end{proof}

\begin{ex}
 Refer to Example \ref{comp} of $K(L_1,L_2,L_3)$ for details of the construction. We will find   $\widetilde{H}^*(\Sigma |lk_{K(L_1,L_2,L_3)}(\sigma')_{I'}|)$ in terms of the links in $K, L_1, L_2, L_3$ for several cases of $\sigma' \in K(L_1,L_2,L_3)$ and  $I' \subset [\sum l_i]$.
\begin{enumerate}
    \item Suppose $\sigma' = \{32\}$. Then $\sigma'_1=\emptyset, \sigma'_2 = \emptyset, \sigma'_3=\{32\}$
    \begin{enumerate}
    \item If $I'=\{11,31\}$, then $I'_1=\{11\}, I'_2=\emptyset, I'_3=\{31\}$. This means $I=\{1,3\}$ and $\sigma = \emptyset$. Thus 
    
    $\begin{array}{cl}\widetilde{H}^*(\Sigma |lk(\sigma')_{I'}|)&=\widetilde{H}^*(\Sigma |lk(\sigma'_1)_{I'_1}|) \otimes  \widetilde{H}^*(\Sigma |lk(\sigma'_3)_{I'_3}|) \otimes \widetilde{H}^*(\Sigma| lk (\sigma)_I| ) \\
    &=\widetilde{H}^*(\Sigma \emptyset) \otimes  \widetilde{H}^*(\Sigma \emptyset) \otimes \widetilde{H}^*(\Sigma| \{\{1\},\{3\} \}| )  \\
    &= 1 \otimes 1 \otimes 1 \otimes \widetilde{H}^*(S^1)
    \end{array}$
    
    This is consistent with $|lk(\sigma')_{I'}|=|\{\{11,21\},\{21,31\}\}_{I'}|=|\{\{11\},\{31\}\}| \simeq S^0$
     \item If $I'=\{31,21\}$, then $I=\{2,3\}$ and $\sigma=\emptyset$. Thus $\widetilde{H}^*(\Sigma |lk(\sigma')_{I'}|)=1 \otimes 1 \otimes 0 $. This is consistent with $|lk(\sigma')_{I'}|=|\{\{31,21\} \} |=\Delta^1$, which is contactible
    \end{enumerate}
    \item Suppose $\sigma'=\{11,32\}$ and $I'=\{31\}$. Then $I=\{3\}$ and $\sigma=\{1\}$. Moreover, $\widetilde{H}^*(\Sigma |lk(\sigma')_{I'}|)=1\otimes 1$, which is consistent with $lk(\sigma')_{I'}=\emptyset$.
\end{enumerate}
\end{ex}

  Recall that Definition \ref{full} of the full subcomplex, $K_I$ or $K|_I$, and that the notation $\mathcal{K}$ denotes the geometric realization of $K$. Specializing Proposition \ref{PPoverCSC} to the case $\xa = (\underline{CA}, \underline{A}) $ we have the following corollary. 
\begin{cor}\label{cor:caa}
$$H^*(Z_{K(L_1,\ldots,L_m)}(\underline{CA}, \underline{A}))$$ $$=\bigoplus\limits_{I\subset [m], J_k\subset [l_k], J_k\neq \emptyset} \widetilde{H}^*(\Sigma \mathcal{K}_I)\otimes \big(\bigotimes\limits_{k\in I} \widetilde{H}^*(\Sigma \mathcal{L}_k|_{J_k})\otimes \widetilde{H}^*(\widehat{A}^{ J_k})\big)$$ where $\widehat{A}^{ J}:=\underset{j\in [J]}{\wedge} A_j$
\end{cor}

\begin{proof}
Recall $L_k$ is a simplicial complex on the set $[l_k]$. For the pair $(CA_i,A_i)$, $B_i=1$, $C_i=0$ and $E_i=\widetilde{H}^*(A_i)$. We have that $Y^{J, \tau}\neq 0$ whenever $\tau =\emptyset$, so that $lk(\tau_k)_{J_k}=L_k|_{J_k}$. Also, if $\sigma \neq \emptyset$, then $Y^{[l_i],\rho'}=0$ because the emptyset is not a nonsimplex of any simplicial complex. Since $\sigma =\emptyset$, it follows that $lk(\sigma)_{I^c}=K_{I^c}$. Recall that $J_k\neq [l_k]$ and hence $J^c_k \neq \emptyset$. 
\end{proof}

An immediate corollary of Corollary \ref{cor:caa} is a computation of the Poincar\'{e} series of  $$\widetilde{H}^*(Z_{K(L_1,\ldots,L_m)}(\underline{CA},\underline{A})).$$

\begin{cor}\label{cor:pscaa}
  $$\bar{P}(\widetilde{H}^*(Z_{K(L_1,\ldots,L_m)}(\underline{CA},\underline{A})))$$ $$=\sum\limits_{I\subset [m], J_k \subset [l_k], J_k\neq \emptyset}\bar{P}(\widetilde{H}^*(\Sigma \mathcal{K}_I)) \prod\limits_{i \in I}\bar{P}(\widetilde{H}^*(\Sigma \mathcal{L}_k|_{J_k}))\bar{P}(\widetilde{H}^*(\widehat{A}^{ J_k}))$$
\end{cor}

\begin{remark}
The Poincar\'{e} series 
$$\bar{P}(\widetilde{H}^*(Z_{K(L_1,\ldots,L_m)}(\underline{CA},\underline{A})))$$ $$=\sum\limits_{B\subset [m]}\bar{P}(\widetilde{H}^*(\Sigma \mathcal{K}_B)) \prod\limits_{b \in B}\bar{P}(\widetilde{H}^*(Z_{L_b}(\underline{CA},\underline{A})))$$
  This generalizes the computation for $(D^2,S^1)$ in Ayzenberg to the case $(\underline{CA},\underline{A})$. The above formulas can be proven instead by also using Ayzenberg's result on the homotopy type of the composition.

\begin{proof}[Alternate proof of Corollary \ref{cor:pscaa}]

Ayzenberg's Lemma 7.5 states that for $J=\cup J_i \subset[\sum{l_i}]$, $K(L_1, \ldots, L_m)_{J}=K_B(L_{b_1}|_{J_{b_1}},\ldots,L_{b_k}|_{J_{b_k}})$ where $B=\{b_i|J_i\neq \emptyset\}$ and Corollary 6.2 states that $K(L_1,\ldots,L_m)\simeq K\ast L_1 \ast \ldots \ast L_m$, and hence $K(L_1,\ldots,L_m) \simeq \Sigma K\wedge L_1 \wedge \ldots \wedge L_m  $ . Now, using the splitting theorem, the wedge lemma, which states that $\widehat{Z}_K(\underline{X},\underline{A}) \simeq \Sigma |K| \wedge \widehat{A}^{ [m]}$, we have the following

\noindent $\widetilde{H}^*(Z_{K(L_1,\ldots,L_m)}(\underline{CA},\underline{A})))=\bigoplus\limits_{J}\widetilde{H}^*(\Sigma K(L_1, \ldots, L_m)_{J}\wedge \widehat{A}^{J})$

$=\bigoplus\limits_{B,J_{b_1},\ldots,J_{b_k}}\widetilde{H}^*( \Sigma K_B(L_{b_1}|_{J_{b_1}},\ldots,L_{b_k}|_{J_{b_k}})\wedge \widehat{A}^{J_{b_1}}\wedge \ldots \wedge \widehat{A}^{J_{b_k}})$\

$=\bigoplus\limits_{B,J_{b_1},\ldots,J_{b_k}}\widetilde{H}^*(\Sigma K_B \wedge \Sigma L_{b_1}|_{J_{b_1}} \wedge \ldots \wedge \Sigma L_{b_k}|_{J_{b_k}}\wedge \widehat{A}^{J_{b_1}}\wedge \ldots \wedge \widehat{A}^{J_{b_k}})$\

$=\bigoplus\limits_{B,J_{b_1},\ldots,J_{b_k}}\widetilde{H}^*(\Sigma K_B)\otimes \bigotimes\limits_{b_i\in B}\widetilde{H}^*(\Sigma  L_{b_i}|_{J_{b_i}} \wedge \widehat{A}^{J_{b_i}})$\

$=\bigoplus\limits_B\widetilde{H}^*(\Sigma K_B)\otimes\bigotimes\limits_{b_i\in B}\bigoplus\limits_{J_{b_i}}\widetilde{H}^*(\Sigma L_{b_i}|_{J_{b_i}} \wedge \widehat{A}^{J_{b_i}})$ \

$=\bigoplus\limits_B\widetilde{H}^*(\Sigma K_B)\otimes\bigotimes\limits_{b_i\in B}\widetilde{H}^*(Z_{L_{b_i}}(\underline{CA},\underline{A}))$

\end{proof}
\end{remark}
\subsection{Multigraded Betti Numbers}\label{mbetti}

 Let $\mathbb{k}$ be the ground field and $\mathbb{k}[m]=\mathbb{k}[v_1, \ldots, v_m]$ be the ring of polynomials in $m$ indeterminates. The ring $\mathbb{k}[m]$ has a $\Z^m$-grading defined by $\deg(v_i)=(0,\ldots, 2, \ldots, 0)$ with $2$ in the $i$-th place. 
Given a free resolution $\ldots \rightarrow R^{-i} \rightarrow \ldots \rightarrow \mathbb{k}[K]$ by $\Z^m$-graded $\mathbb{k}[m]$ modules, we have the Tor-module 
$$ \text{Tor}_{\mathbb{k}[m]}(\mathbb{k}[K],\mathbb{k})= \bigoplus_{i \in \Z_{\geq 0}, \overline{j} \in \Z^m} \text{Tor}^{-i,2\overline{j}}_{\mathbb{k}[m]}(\mathbb{k}[K], \mathbb{k})  $$

The multigraded betti numbers of a simplicial complex are then defined in terms of the muligraded structure of the Tor-module. The $(-i,2\overline{j})$-th betti number of $K$ is the dimension of $\text{Tor}^{-i,2\overline{j}}_{\mathbb{k}[m]}(\mathbb{k}[K], \mathbb{k})$ over $\mathbb{k}$: 
$$\beta_\mathbb{k}^{-i,2\overline{j}}(K):= \dim_\mathbb{k}(\text{Tor}^{-i,2\overline{j}}_{\mathbb{k}[m]}(\mathbb{k}[K],\mathbb{k})).$$
Since the  cohomology ring of the moment-angle complex is  $\text{Tor}_{\mathbb{k}[m]}(\mathbb{k}[K],\mathbb{k})$, the multigraded betti numbers of a simplicial complex have a topological interpretation in terms of $Z_K(D^2,S^1)$ \cite{MR1897064}. We adapt this interpretation for an arbitrary polyhedral product of pairs $(\underline{CA},\underline{A})$.

Morever, taking into account Hochster's theorem:
$$ \beta_\mathbb{k}^{-i,2\overline{j}}(K) = \dim \tilde{H}^{|J|-i-1}(K_J;\mathbb{k}) $$ 
where there is a straightforward association between $\overline{j}\in \Z^m$ and a subset $J\subset [m]$.
Then using the splitting theorem (\ref{splitting}) and wedge lemma we give an analogous definition for the multigraded Betti numbers of a polyhedral product space. 
\begin{defn}\label{betti}For $i\in \Z$ and $J\subset [m]$, the \emph{multigraded Betti numbers} of $Z_K(\underline{CA},\underline{A})$,
denoted $\beta^{i,J}(Z_K(\underline{CA},\underline{A}))$,
are defined as  $$\beta^{i,J}(Z_{K}(\underline{CA},\underline{A})):= \dim\widetilde{H}^{i}(\Sigma \mathcal{K}_J \wedge \widehat{A}^J)$$
Let $s$ and $t_1, \ldots, t_m$ be indeterminates such that $\bar{t}^J=t_1^{j_1} \ldots t_m^{j_m}$ where $j_i=1$ if $i \in J$ and $j_i=0$ otherwise. The \emph{beta-polynomial} of $Z_K(\underline{CA},\underline{A})$ is defined as 
 $$\beta_{Z_K(\underline{CA},\underline{A})}(s,\bar{t}):= \sum_{i \in \Z, J\subset [m]} {\beta}^{i,J}(Z_{K}(\underline{CA},\underline{A})) s^{i}\bar{t}^J$$
 and the reduced beta-polynomial $$\widetilde{\beta}_{Z_K(\underline{CA},\underline{A})}(s,\bar{t}):= {\beta}_{Z_K(\underline{CA},\underline{A})}(s,\bar{t}) -1 = \sum_{i \in \Z, \emptyset \neq J\subset [m]} {\beta}^{i,J}(Z_{K}(\underline{CA},\underline{A})) s^{i}\bar{t}^J$$
\end{defn}
The definition of multigraded Betti numbers of a simplicial complex is given in terms of a Tor algebra and the Stanley-Reisner ring, which are in terms of indeterminates in degree $2$. When considering he $(-i,2j)$-th Betti number of $K$ and applying Hochster's theorem, the $(-i,2j)$-th Betti number should be the dimension of the cohomology of $K_J$ in degree $(-i + 2j) - j - 1 = -i + j - 1$ where $|J|=j$. This is equivalent to Definition \ref{betti} since $\widetilde{H}^{i}(\Sigma \mathcal{K}_J \wedge \widehat{A}^J)=\widetilde{H}^{i-j-1}(K_J)$, with a change of variables for the shift in cohomological degree.
\begin{prop}\label{beta} The multigraded Betti numbers of $Z_{K(L_1,\ldots,L_m)}(\underline{CA},\underline{A})$ can be expressed in terms of the multigraded Betti numbers of the polyhedral products associated to each of the simplicial complexes $K, L_1, \ldots, L_m$. Their beta-polynomials have the following relationship.
$$\beta_{Z_{K(L_1,\ldots L_m)}(\underline{CA},\underline{A})}(s,\bar{t})=\beta_{Z_K(D^1,S^0)}(s,\widetilde{\beta}_{Z_{L_1}(\underline{CA},\underline{A})}(s,\bar{t}),\ldots ,\widetilde{\beta}_{Z_{L_m}(\underline{CA},\underline{A})}(s,\bar{t}))$$
\end{prop}
\begin{proof}
Let $i'\in \mathbb{Z}$, $J=\cup J_i \subset[\sum{l_i}]$, $B=\{b_i\in [m]|J_i\neq \emptyset\}=\{b_1,\ldots,b_k\}$,  $n+p=i'$, $r+\sum r_s=n$, $\sum c_s=p$. Using Corollary \ref{cor:caa}, we have  

\noindent $\beta_{Z_{K(L_1,\ldots,L_m)}(\underline{CA},\underline{A})}(s,\bar{t})=\sum\limits_{i',J}\dim\widetilde{H}^{i'}(\Sigma K(L_1,\ldots,L_m)_J \wedge \widehat{A}^J) s^{i'}\bar{t}^J$

$=\sum\limits_{\substack{i',B,\\ J_{b_1},\ldots,J_{b_k}}}\sum\limits_{n,p}\dim\widetilde{H}^n(\Sigma K(L_1,\ldots,L_m)_J)\dim\widetilde{H}^p(\widehat{A}^{J_{b_1}}\wedge \ldots \wedge \widehat{A}^{J_{b_k}}) s^{i'}\bar{t}^{J}$

 $=\sum\limits_{\substack{i',B,\\ J_{b_1},\ldots,J_{b_k}}}\sum\limits_{n,p}(\sum\limits_{r,r_1,\ldots,r_k} \dim \widetilde{H}^r(\Sigma K_B)\prod\limits_{i=1}^{k}\dim \widetilde{H}^{r_i}(\Sigma L_{b_i}|_{J_{b_i}}))(\sum\limits_{c_1,\ldots,c_k}\prod\limits_{i=1}^k \dim \widetilde{H}^{c_i}(\widehat{A}^{J_i}))s^{i'}\bar{t}^{J}$

 $=\sum\limits_{\substack{i',B,\\ J_{b_1},\ldots,J_{b_k}}}\sum\limits_{n,p}\sum\limits_r \dim\widetilde{H}^r(\Sigma K_B) \prod\limits_{i=1}^k \sum\limits_{r_1,\ldots, r_k, c_1,\ldots, c_k}\dim\widetilde{H}^{r_i+c_j}(\Sigma L_{b_i}|_{J_{b_i}} \wedge \widehat{A}^{J_{b_i}})s^{i'}\bar{t}^{J_{b_1}}\ldots\bar{t}^{J_{b_k}}$

 Next, we will use a change of variables in order to rewrite this in a recognizable form, let $$ r_i+c_j=a_{i+j}$$
Then $i'=n+p=r+\sum r_i + \sum c_i=r+\sum a_{i+j}=r+ \sum a_u$.
Therefore, 

$\beta_{Z_{K(L_1,\ldots,L_m)}(\underline{CA},\underline{A})}(s,\bar{t})$

\noindent $=\sum\limits_{i',B, J_{b_1},\ldots,J_{b_k}}\sum\limits_{n,p}\sum\limits_r \dim\widetilde{H}^r(\Sigma K_B) \prod\limits_{i=1}^k \sum\limits_{r_1,\ldots, r_k, c_1,\ldots, c_k}\dim\widetilde{H}^{r_i+c_j}(\Sigma L_{b_i}|_{J_{b_i}} \wedge \widehat{A}^{J_{b_i}})s^{i'}\bar{t}^{J_{b_1}}\ldots\bar{t}^{J_{b_k}}$

\noindent $=\sum\limits_{B,r} \dim\widetilde{H}^{r}(\Sigma K_B) s^r \prod\limits_{i=1}^k \sum\limits_{J_b,a_u} \dim \widetilde{H}^{a_u}(\Sigma L_{b_i}|_{J_{b_i}} \wedge \widehat{A}^{J_{b_i}})s^{a_u}\bar{t}^{J_s}$

\noindent $=\beta_{Z_K(D^1,S^0)}(s,\widetilde{\beta}_{Z_{L_1}(\underline{CA},\underline{A})}(s,\bar{t}),\ldots ,\widetilde{\beta}_{L_m(\underline{CA},\underline{A})}(s,\bar{t}))$
\end{proof}

\begin{ex}
Consider the composed simplicial complex from example \ref{bdayhat}. We need to compute the reduced beta-polynomials of each complex $L_i$. Since the full subcomplex of $L_1$ associated to $\{11\}$ is the empty set, $$\widetilde{\beta}_{Z_{L_1}(\underline{CA},\underline{A})}=\sum \dim \widetilde{H}^i(\Sigma \emptyset \wedge A_{11})s^i t_{11}$$
The full subcomplexes of $L_2$ are all contractible, so its beta-polynomial is the zero polynomial. The only non-trivial full subcomplex of $L_3$ is associated to $\{31,32\}$, and hence its reduced beta-polynomial is $$\widetilde{\beta}_{Z_{L_3}(\underline{CA},\underline{A})}=\sum \dim \widetilde{H}^i(\Sigma \partial \Delta^1 \wedge A_{31} \wedge A_{32})s^i t_{31} t_{32}$$
The non-contractible full subcomplexes of $K$ are $\{1,2\},\{1,3\},\{1,2,3\}$. Since the beta-polynomial of $L_2$ is zero, any subsets of $[3]$ that contain $2$ do not contribute any non-trivial terms. Apply Proposition \ref{beta}.

\noindent $\beta_{Z_{K(L_1,\ldots L_m)}(\underline{CA},\underline{A})}(s,\bar{t})$
\begin{align} 
\quad &=\beta_{Z_K(D^1,S^0)}(s,\widetilde{\beta}_{Z_{L_1}(\underline{CA},\underline{A})}(s,\bar{t}),\ldots ,\widetilde{\beta}_{Z_{L_m}(\underline{CA},\underline{A})}(s,\bar{t})) \nonumber \\
\quad &=  \sum \limits_{i \in \Z, J\subset [m]} \dim\widetilde{H}^{i}(\Sigma \mathcal{K}_J) s^{i}(\widetilde{\beta}_{Z_{\underline{L}}(\underline{CA},\underline{A})})^J  \nonumber \\
\quad &=
1+\sum \dim H^i(\Sigma \partial \Delta^1)s^i (\widetilde{\beta}_{Z_{L_1}(\underline{CA},\underline{A})}) (\widetilde{\beta}_{Z_{L_3}(\underline{CA},\underline{A})})
 \nonumber \\
\quad  &=
 1+s(\sum \dim \widetilde{H}^i(\Sigma \emptyset \wedge A_{11})s^i t_{11})(\sum \dim \widetilde{H}^i(\Sigma \partial \Delta^1 \wedge A_{31} \wedge A_{32})s^i t_{31} t_{32}) \label{right}
\end{align}
Since all full subcomplexes of $K(L_1,L_2,L_3)$ are contractible except those associated to the empty set and and the set $\{11,31,32\}$, its  beta-polynomial is 
 \begin{align}
{\beta}_{Z_{K(L_1,L_2,L_3)}(\underline{CA},\underline{A})}(s,\bar{t})&= \sum \limits_{i \in \Z, J} \dim\widetilde{H}^{i}(\Sigma \mathcal{K}_J \wedge \widehat{A}^J) s^{i}\bar{t}^J \nonumber \\
&=1+\sum \dim \widetilde{H}^i(\Sigma \partial \Delta^2 \wedge A_{11} \wedge A_{31}  \wedge A_{32})s^i t_{11} t_{31} t_{32} \label{left}
\end{align}

If for example $A_i=S^2$ for all $i$, then both expressions \ref{left} and \ref{right} simplify to $$1+s^8 t_{11} t_{31} t_{32}$$
\end{ex}

\subsection{Ring Structure} Methods from \cite{BBCG4} can also be used to describe the ring structure of a polyhedral product $Z_K(\underline{CA},\underline{A})$ in terms of the cohomology ring of the decomposition given in Theorem \ref{BBCGSS}. Given generators 
\begin{center}
$ \alpha=n_\alpha \otimes a_1 \otimes \ldots \otimes a_m$

$\gamma = n_\gamma \otimes g_1 \otimes \ldots \otimes g_m $ \end{center}
 where $n_\alpha \in \widetilde{H}^*(\Sigma | lk(\sigma_1)|_{I_1^c})$, $n_\gamma \in \widetilde{H}^*(\Sigma |lk(\sigma_2)|_{I_2^c})$, the remaining factors are in the appropriate $E_i, B_i$ or $C_i$. For instance, for every $i\in \sigma_1$, $a_i \in C_i$ and so on. The cup product $\alpha \smallsmile \gamma$ is described in terms of coordinate-wise multiplication and a pairing $$ \widetilde{H}^*(\Sigma | lk(\sigma_1)|_{I_1^c}) \otimes \widetilde{H}^*(\Sigma | lk(\sigma_2)|_{I_2^c}) \rightarrow \widetilde{H}^*(\Sigma | lk(\sigma_3)|_{I_3^c})$$ where $I_3$ and $\sigma_3$ are described in terms of $I_1, I_2,\sigma_1,\sigma_2$. A complication that arises is that in the product the indexing set of the $C_i$'s could be larger. In the decomposition (from Theorem \ref{BBCGSS}), the $C_i$'s in every term are indexed by a simplex in $K$. Therefore, if the larger indexing set of the $C_i$'s does not correspond to a simplex in $K$, the cup product must be zero. We may think of $H^*(Z_K\xa)$ as living in the larger tensor algebra modulo the generalized Stanley-Reisner ideal of $K$.

Recall the generalized Stanley-Reisner ideal $I(K)$ in $\widetilde{H}^*(X_1) \otimes \ldots \otimes \widetilde{H}^*(X_m)$, $$I(K)=\langle x_{i_1} \otimes \ldots \otimes x_{i_k} \: | \: \{i_1, \ldots, i_k\} \notin K \rangle.$$
Since the generalized Stanley-Reisner ideal is essential to understanding the ring structure of $H^*(Z_K\xa)$, we will describe the generalized Stanley-Reisner ideal in the case that the underlying simplicial complex is a composition $K(L_1, \ldots, L_m)$ in terms of the generalized Stanley-Reisner ideal of $K, L_1, \ldots, L_m$.
\begin{prop}
The generalized Stanley-Reisner ideal of $K(L_1, \ldots, L_m)$ is the generalized Stanley-Reisner ideal of $K$ in terms of the generalized Stanley-Reisner ideals of $L_1, \ldots, L_m$  $$I(K(L_1, \ldots, L_m))=\langle c_{i_1}\otimes \ldots \otimes c_{i_k} \: | \: \{i_1, \ldots, i_k\} \notin K \rangle$$
where $c_{i}\in I(L_i)$ for $i_1 \leq i \leq i_k$.
\end{prop}
\begin{proof}
Recall from the proof of Theorem \ref{PPoverCSC} that $C_i=I(L_i)$.

Suppose $c^\sigma \in I(K(L_1, \ldots, L_m))$ where $\sigma \subset [\sum l_i ]$ such that $\sigma \notin K(L_1, \ldots, L_m)$.
Recall from the equivalent definition of $K(L_1, \ldots L_m)$ (after Definition \ref{comp}) that $\sigma$ is of the form
$$\sigma = \bigcup\limits_{i\in [m]}  \sigma_i$$
where $\sigma_i \subset [l_i]$ and $\{ i \in [m] \: | \: \sigma_i \notin L_i \} \notin K$. Let $A=\{ i \in [m] \: | \: \sigma_i \notin L_i \}=\{i_1,\ldots,i_k\}$. In other words $c^\sigma=c^A=c_{i_1}\otimes \ldots \otimes c_{i_k}$ where $A \notin K$ and $c_i \in I(L_i)$. It follows that $c^\sigma \in \langle c_{i_1}\otimes \ldots \otimes c_{i_k} \: | \: \{i_1, \ldots, i_k\} \notin K \rangle$.

Suppose $c=c_{i_1}\otimes \ldots \otimes c_{i_k} \in \langle c_{i_1}\otimes \ldots \otimes c_{i_k} \: | \: \{i_1, \ldots, i_k\} \notin K \rangle$
where $c_{i}\in I(L_i)$ for $i_1 \leq i \leq i_k$. Since $c_i \in I(L_i)$, it is of the form $c_i=c^{\tau_i}$ for some $\tau_i \notin L_i$. Since $\{ i_1, \ldots, i_k \} \notin K$, we have that $\tau = \cup \tau_i \notin K(L_1,\ldots, L_m)$. Therefore, $c=c^\tau \in I(K(L_1,\ldots, L_m))$.
\end{proof}

\begin{ex}
The generalized Stanley-Reisner ideal is generated by the minimal non-faces, a set that is not a simplex but every subset is. Suppose we have the following simplicial complexes.
\begin{center}
$\begin{array}{lll}
L_1 \subset [2] & \qquad L_1=\{\{11\},\{12\}\} & \qquad I(L_1)=\langle c_{11} \otimes c_{12} \rangle \\
L_2 \subset [2] & \qquad L_2=\{ \{21\} \} & \qquad I(L_2)= \langle c_{22} \rangle \\
K \subset [2] & \qquad K=\{\{2\}\} & \qquad I(K)=\langle c_1 \rangle = \langle c_{11} \otimes c_{12} \rangle
\end{array}$
\end{center}
The composition $K(L_1,L_2)=L_1 \ast \Delta^1=\{\{11,21,22\},\{12,21,22\}\}$ and hence $$I(K(L_1,L_2))=\langle c_{11}\otimes c_{12} \rangle$$.
\end{ex}

\section{The pair $(L_i,\emptyset)$} \label{pair}
In this section we will find a formula for the cohomology groups of $Z_{Z^*_K(L_i, \emptyset)}(\underline{X},\underline{A})$, the polyhedral product over a polyhedral join given by the pairs $(L_i,\emptyset)$. In this case, we get a similar formula to Theorem \ref{aa}
\begin{equation}\label{join}
Z_{Z^*_K(L_i,\emptyset)}(\underline{X}, \underline{A})=Z_K(Z_{L_i}(\underline{X},\underline{A}),\prod_{j\in [l_i]}{A_j})
\end{equation} 
As an application, we can write the polyhedral product $Z_K(S^n,\vee S^0)$ as the real moment-angle complex $Z_{Z^*_K(\partial \Delta^{n_i},\emptyset)}(D^1,S^0)$.

It follows from the discussion in  Section \ref{general} that the kernel, cokernel and image can be computed if the links of simplices in $L_i$ can be described in general. Note that $L_i$ and its subsimplicial complex, $\emptyset$, do not have any (nontrivial) simplices in common, so the links do not present any issues. Equation \ref{join} and Theorem \ref{BBCGSS} imply the following formula.
\begin{thm} \label{empty}
Given simplicial complexes $L_i$ on the vertex sets $[l_i]$ with no ghost vertices and pairs $(X_{ij},A_{ij})$, where $i$ varies in $[m]$ and $j$ varies in $[l_i]$, that satisfy the freeness condition of Definition \ref{free} with decompositions   $$H^*(X_{ij})=B_{ij}\oplus C_{ij}$$ $$H^*(A_{ij})=B_{ij}\oplus E_{ij}$$
 Then we have

    $H^*(Z_{Z^*_K(L_i,\emptyset)}(\underline{X}, \underline{A}))=$
$$\bigoplus \limits_{J, \tau, I, \sigma}  E^T \otimes B^{(T\cup S)^c} \otimes C^{S} \otimes (\bigotimes \limits_{v \in \tau} \widetilde{H}^*(\Sigma |lk(\sigma)|_I)) \otimes \widetilde{H}^*(\Sigma |lk(\tau)|_{J^c}) $$
 where 
\begin{enumerate}
\item $J \subset [m]$ with a simplex $\tau$ of $K$ such that $\tau \subset J$ 
 \item For $v \in \tau$, take subsets $I_v \subset [l_v]$ and a simplex $\sigma \in L_v$ such that $\sigma \subset I^c$. For $k \in [m] \backslash J$, consider subsets $I_k \subset [l_k]$. Then $T$ and $S$ are defined by $$ T=\displaystyle (\bigcup_{[m]\backslash J} I_k) \cup (\bigcup_{v\in \tau} I_v)$$ $$ S= \bigcup_{v \in \tau} \sigma_v$$
\end{enumerate}

\end{thm}
\begin{proof}
From Definition \ref{free}, we need to find the kernel, $E_i$, image, $B_i$, and cokernel, $C_i$, of the map $$H^*(Z_{L_i}(\underline{X},\underline{A}))  \rightarrow  H^*(\prod_{j\in [l_i]}{A_j})$$ where $H^*(\prod_{j\in [l_i]}{A_j})= E_i \oplus B_i$ and $H^*(Z_{L_i}(\underline{X},\underline{A})) = C_i \oplus B_i$.
Since the cohomology of each space in the pair is given by

$\begin{array}{lcl}
H^*(\prod_{j\in [l_i]}{A_j})&=&\bigoplus \limits_{I \subset [l_i]} B^I \otimes E^{I^c}\\
%
H^*(Z_{L_i}(\underline{X},\underline{A}))&=&\bigoplus\limits_{\substack{\sigma\subset I\subset [m] \\ \sigma \in L_i}} E^{I^c} \otimes \widetilde{H}^*(\Sigma | lk(\sigma)|_{I^c}) \otimes Y^{I,\sigma}\\
          &=&\big( \bigoplus \limits_{I\subset [m]}E^{I^c}\otimes B^I \otimes \widetilde{H}^*(\Sigma \mathcal{L}_i|_{I^c}) \big)  \\
        & & \multicolumn{1}{c}{ \oplus } \\
				 & & \big (\bigoplus\limits_{\substack{ \sigma\subset I\subset [m] \\ \emptyset \neq \sigma \in L_i}} E^{I^c} \otimes \widetilde{H}^*(\Sigma |lk(\sigma)|_{I^c}) \otimes Y^{I,\sigma} \big ) \end{array} $
				
				\
				
\noindent and the full subcomplex $ {\mathcal{L}_i}|_{I^c}$ is only empty when $I^c=\emptyset$ (because $L_i$ has no ghost vertices), we have that 
 \begin{align*}
B_i&=B^{[l_i]}\\
E_i&= \bigoplus \limits_{I \subsetneq [l_i]} B^I \otimes E^{I^c}\\
C_i&=\bigoplus\limits_{\substack{\sigma\subset I\subsetneq [m] \\ \sigma \in L_i \\ \emptyset \neq \sigma \in L_i}} E^{I^c} \otimes \widetilde{H}^*(\Sigma | lk(\sigma)|_{I^c}) \otimes Y^{(I,\sigma)}\end{align*}

\noindent where ``$\sigma, I \neq \emptyset$'' means that $\sigma$ and $I$ are not both the empty set. Then substituting,

 \noindent $H^*(Z_{Z^*_K(L_i,\emptyset)}(\underline{X}, \underline{A}))$
 
 $\begin{array}{ccll}
\quad &=
& \multicolumn{2}{l}{\bigoplus \limits_{J, \tau}
\widetilde{H}^*(\Sigma |lk(\tau )|_{J^c}) \otimes E^{[m]\backslash J} \otimes C^\tau \otimes B^{J\backslash \tau} } 
\\
\quad &=
        &\bigoplus \limits_{J, \tau} \Big( \widetilde{H}^*(\Sigma |lk(\tau)|_{J^c}) 
            & \otimes \bigotimes \limits_{k \in [m]\backslash J} (\bigoplus \limits_{L\subset [l_k]} E^L \otimes B^{L^c})\\
            & & & \otimes \bigotimes \limits_{v \in \tau} \big( \bigoplus \limits_{\substack{ I \subset [l_v],\\  \sigma \in L_v,  \sigma, I \neq \emptyset}} E^I \otimes Y^{I^c,\sigma} \otimes \widetilde{H}^*(\Sigma |lk(\sigma)|_I)\big) 
            \\
            & & & \otimes \bigotimes \limits_{ u \in J\backslash \tau} B^{[l_u]} \quad \Big) 
            \\
\quad &=
        &\bigoplus \limits_{\substack{J, \tau, \\ I, L, \sigma}} \Big(  \widetilde{H}^*(\Sigma |lk(\tau)|_{J^c})
            & \bigotimes \limits_{k \in [m]\backslash J} ( E^L \otimes B^{L^c}) 
            \\
            & & & \otimes \bigotimes \limits_{v \in \tau} \big( E^I \otimes Y^{I^c,\sigma} \otimes \widetilde{H}^*(\Sigma |lk(\sigma)|_{I}) \big) 
            \\
            & & & \otimes \bigotimes \limits_{u \in J\backslash \tau} B^{[l_u]} \quad \Big)
\end{array}$

\end{proof}

\begin{cor}
Suppose $(X_{ij},A_{ij})=(CA_{ij},A_{ij})$. Then $$H^*(Z_{Z^*_K(L_i,\emptyset)}(\underline{CA},\underline{A}))=\bigoplus \limits_{J,\tau, I_k,I_v \neq \emptyset} \big( \widetilde{H}^*(\widehat{A}^{T}) \otimes (\bigotimes \limits_{v \in \tau} \widetilde{H}^*(\Sigma \mathcal{L}_v|_{I_v})) \otimes \widetilde{H}^*(\Sigma |lk(\tau)|_{J^c}))\big)$$  
\noindent where 
\begin{enumerate} 
\item $J \subset [m]$ with a simplex $\tau$ of $K$ such that $\tau \subset J$ 
\item For $k \in [m]\backslash J$ and $v \in \tau$, take subsets $I_v \subset [l_v]$ and $I_k \subset [l_k]$, and define $T$:
$$T:= (\bigcup_{k\in [m]\backslash J} I_k) \cup (\bigcup_{v\in \tau} I_v)$$
\end{enumerate} 
\end{cor}
\begin{proof}
With the given pairs, we know that $C_{ij}=0$, $B_{ij}=1$ and $E_{ij}=\widetilde{H}^*(A_{ij})$ for all $ij$. Since $C_{ij}=0$, $\sigma$ in Proposition \ref{empty} must be the empty set. Therefore $I$ cannot be the empty set and

\noindent $H^*(Z_{Z^*_K(L_i,\emptyset)}(\underline{CA},\underline{A}))$

$=\bigoplus \limits_{\substack{J,\tau, I_k, \\ I_v \neq \emptyset}} \big( (\bigotimes \limits_{k\in [m] \backslash J} E^{I_k} \otimes B^{I_k^c} ) \otimes (\bigotimes \limits_{v \in \tau} E^{I_\alpha}\otimes B^{I_v^c} \otimes \widetilde{H}^*(\Sigma L_v|_{I_v})) \otimes (\bigotimes \limits_{j \in J \backslash \tau}B^{[l_\beta]}) \otimes \widetilde{H}^*(\Sigma |lk(\tau)|_{J^c}) \big)$

$=\bigoplus \limits_{J,\tau, I_k,I_\alpha \neq \emptyset} \big( (\bigotimes \limits_{k\in [m] \backslash J} E^{I_k} )\otimes (\bigotimes \limits_{v \in \tau} E^{I_v} \otimes \widetilde{H}^*(\Sigma L_v|_{I_v}))  \otimes \widetilde{H}^*(\Sigma |lk(\tau)|_{J^c})\big)$
\end{proof}

\bibliographystyle{amsplain}
\bibliography{mybibliography}
\end{document}